\documentclass[10pt, a4paper]{amsart}
\usepackage{amssymb, amsmath, amsfonts, mathrsfs, amsthm}
\usepackage[utf8]{inputenc}
\usepackage{lmodern}
\usepackage[T1]{fontenc}
\usepackage[a4paper,includeheadfoot,margin=2.54cm]{geometry}
\usepackage[matrix, arrow, curve]{xy}

\DeclareMathOperator{\PP}{\mathbb{P}}

\DeclareMathOperator{\Bim}{\mathrm{Bim}}
\DeclareMathOperator{\Bir}{\mathrm{Bir}}
\DeclareMathOperator{\Aut}{\mathrm{Aut}}
\DeclareMathOperator{\Psaut}{\mathrm{Psaut}}

\numberwithin{equation}{section}

\begin{document} 

\title{Finite abelian subgroups in the groups of birational and bimeromorphic selfmaps}
\date{}
\author{Aleksei Golota} 
\thanks{This work was supported by the Russian Science Foundation under grant no. 23-11-00033,
https://rscf.ru/en/project/23-11-00033/}

\newtheorem{theorem}{Theorem}[section] 
\newtheorem*{ttheorem}{Theorem}
\newtheorem{lemma}[theorem]{Lemma}
\newtheorem{proposition}[theorem]{Proposition}
\newtheorem*{conjecture}{Conjecture}
\newtheorem{corollary}[theorem]{Corollary}

{\theoremstyle{remark}
\newtheorem{remark}[theorem]{Remark}
\newtheorem{example}[theorem]{Example}
\newtheorem{notation}[theorem]{Notation}
\newtheorem{question}[theorem]{Question}
}

\theoremstyle{definition}
\newtheorem{construction}[theorem]{Construction}
\newtheorem{definition}[theorem]{Definition}

\begin{abstract} Let $X$ be a complex projective variety. Suppose that the group of birational automorphisms of $X$ contains finite subgroups isomorphic to $(\mathbb{Z}/N\mathbb{Z})^r$ for $r$ fixed and $N$ arbitrarily large. We show that $r$ does not exceed $2\dim(X)$. Moreover, the equality holds if and only if $X$ is birational to an abelian variety. We also show that an analogous result holds for groups of bimeromorphic automorphisms of compact K\"ahler spaces under some additional assumptions.
\end{abstract}

\maketitle

\section{Introduction} In the present paper we study finite abelian subgroups in the groups of birational automorphisms of projective algebraic varieties (over a field of zero characteristic), or in the groups of bimeromorphic automorphisms of compact K\"ahler spaces. The starting point for us is the following recent theorem by I. Mundet i Riera \cite[Theorem ~1.9]{Mun21}.

\begin{theorem} \label{mundet1} Let $X$ be a connected compact K\"ahler manifold. Suppose that there exists $r \in \mathbb{N}$ such that for arbitrarily large positive integers $N$ the group $\Aut(X)$ contains a subgroup isomorphic to $(\mathbb{Z}/N\mathbb{Z})^r$. Then $\Aut(X)$ contains a subgroup isomorphic to a compact real torus of dimension $r$. In addition, $r \leqslant 2\dim(X)$, and if $r = 2\dim(X)$ then $X$ is biholomorphic to a compact complex torus.
\end{theorem}

The maximal number $r$ satisfying the assumptions of Theorem \ref{mundet1} is called in \cite{Mun21} the (holomorphic) {\em discrete degree of symmetry} of $X$. More generally, in \cite{Mun21} I. Mundet i Riera defines and studies this invariant for continuous group actions on topological manifolds. In some cases, the discrete degree of symmetry can be compared to the maximal dimension of a torus acting effectively on a manifold \cite[Theorem 1.7]{Mun21}. In connection with Theorem \ref{mundet1} I. Mundet i Riera also asks whether the same bound on $r$ holds also for birational automorphism groups. In fact, this invariant has implicitly appeared in the study of $p$-subgroups of birational automorphism groups. For instance, in \cite[Theorem 2.9]{Xu18} J. Xu proved the following result for non-uniruled algebraic varieties.

\begin{theorem}\label{jxu1} Let $X$ be a non-uniruled algebraic variety over an algebraically closed field of characteristic zero. There exists a constant $b(X)$ such that the group $\Bir(X)$ contains an element of order greater than $b(X)$ if and only if $X$ is birational to a variety $X'$ which admits an effective action of an abelian variety.
\end{theorem}

A remarkable result of Yu. Prokhorov and C. Shramov \cite[Theorem 1.10]{PS16}, together with C. Birkar's solution of the BAB conjecture \cite[Theorem 1.1]{Bir21}, provides a stronger bound for rationally connected varieties.

\begin{theorem}\label{prokshr} Let $X$ be a rationally connected algebraic variety of dimension $n$ over an algebraically closed field of characteristic zero. There exists a constant $L = L(n)$ such that, for any prime number $p > L(n)$, each finite $p$-subgroup $G \subset \Bir(X)$ is isomorphic to $(\mathbb{Z}/p\mathbb{Z})^r$ for some $r \leqslant n$. 
\end{theorem}

By a result of J. Xu \cite{Xu20} the constant $L$ in the above theorem can be taken to be $n+1$. Moreover, J.\,Xu proved a rationality criterion for rationally connected varieties admitting an action of ~$(\mathbb{Z}/p\mathbb{Z})^r$ in terms of $r$ and $p$ (see \cite[Theorem 4.5]{Xu18}).

\begin{theorem}\label{jxu2} Let $X$ be a rationally connected algebraic variety of dimension $n$ over an algebraically closed field of characteristic zero. Then there exists a constant $R(n)$ such that if $\Bir(X)$ contains a subgroup isomorphic to $(\mathbb{Z}/p\mathbb{Z})^n$ for some $p > R(n)$ then $X$ is rational.
\end{theorem}

Informally speaking, these results suggest that existence of finite abelian subgroups in $\Bir(X)$ of unbounded orders should imply existence of algebraic groups (of positive dimension depending on $r$) acting on $X$ by birational automorphisms, at least if the ranks $r$ of the finite abelian groups are close to maximal. For smaller values of $r$ the relation between finite abelian and algebraic subgroups of $\Bir(X)$ is more delicate. For instance, there exists a sequence of finite cyclic subgroups of $\mathrm{Cr}_2(\mathbb{C}) = \Bir(\mathbb{P}^2_{\mathbb{C}})$ which generate a subgroup isomorphic to $\mathbb{Q}/\mathbb{Z}$ but are not contained in any torus in the Cremona group \cite{Wri79}. Existence of finite abelian subgroups of unbounded orders in the group $\Bir(X)$, where $X$ is a non-rational rationally connected threefold, is a difficult open problem (see \cite[Question 4.8]{PS18}). Another related open problem (cf. \cite[Conjecture 1.7]{Xu18}) is a conjectural description of projective varieties with non-Jordan groups of birational automorphisms. The first examples of such varieties were constructed in \cite{Zar14}; a complete description exists in dimension 3 by \cite[Theorem 1.8]{PS18} and \cite[Theorem 1.6]{Xu18}.

We should also mention a ``toroidalization principle'' recently studied by J. Moraga in his works on Kawamata log terminal singularities \cite{Mor20, Mor21a, Mor21b}. In particular, he showed that existence of ``large'' finite abelian groups of rank $n$ acting on a projective Fano type variety of dimension $n$ implies that $X$ is birational to a log Calabi--Yau toric pair (\cite[Theorem 2]{Mor20}). In \cite[Theorem 1]{Mor21b} a general result on toroidalization for finite group actions on klt singularities is proved. The case of cyclic group actions on Fano type surfaces is studied in \cite{Mor21a}.

The aim of this paper is to initiate a systematic study of an invariant similar to the discrete degree of symmetry for groups of birational (and bimeromorphic) automorphisms. Our main result is a generalization of Theorem ~\ref{mundet1} to groups of birational automorphisms.

\begin{theorem}\label{Main} Let $X$ be a projective algebraic variety over an algebraically closed field of zero characteristic. Suppose that there exists an unbounded sequence $\{N_i\}_{i \in \mathbb{N}}$ of positive integers such that the group $\Bir(X)$ contains subgroups isomorphic to $(\mathbb{Z}/N_i\mathbb{Z})^r$ for some fixed $r$. Then $r \leqslant 2\dim(X)$, and in case of equality $X$ is birational to an abelian variety. 
\end{theorem}

Compared to Theorem \ref{jxu1}, we consider also uniruled varieties; moreover, we do not assume that the orders $N_i$ of generators of the finite groups are prime. The main idea of the proof is to consider the action of $\Bir(X)$ on the maximal rationally connected (MRC) fibration of $X$ (see Definition \ref{mrcdef} below); this idea is already present in J. Xu's work (see \cite[Proposition 2.12]{Xu18}). Combining it with some technical results from our paper \cite{Gol23}, we prove an analogous result for groups of bimeromorphic selfmaps of compact K\"ahler spaces. We have to assume the existence of quasi-minimal models (see Definition \ref{qmin} and Proposition \ref{NonUniKah} below) for the (non-uniruled) base of the MRC fibration of $X$. This is the case if the base has dimension at most 3 by \cite[Theorem 1.1]{HP16} and is expected to be true in any dimension.

\begin{theorem} \label{MainKahler} Let $X$ be a compact K\"ahler space. Assume that the base $B$ of the MRC fibration of $X$ admits a quasi-minimal model. Suppose that there exists an unbounded sequence $\{N_i\}_{i \in \mathbb{N}}$ of positive integers such that the group $\Bim(X)$ contains subgroups isomorphic to $(\mathbb{Z}/N_i\mathbb{Z})^r$ for some fixed $r$. Then we have $r \leqslant 2\dim(X)$ and in case of equality $X$ is bimeromorphic to a compact complex torus. 
\end{theorem}

Let us outline the structure of the paper. In Section 2 we gather some technical results. Section 3 is devoted to the proof of our main theorem. First, in subsection 3.1 we use techniques from our previous paper \cite{Gol23} to generalize Theorem \ref{mundet1} to pseudoautomorphisms of compact K\"ahler spaces with rational singularities (see Theorem \ref{MainThm2}). Then, in subsection 3.2 we prove the main theorem for non-uniruled projective varieties (Theorem \ref{NonUni}), following the ideas from \cite[Section 2]{Xu18}. In subsection 3.3 we use the results of Prokhorov and Shramov from \cite{PS16} to derive the bound on $r$ for abelian groups acting on rationally connected varieties. Finally, in subsection 3.4 we derive Theorem \ref{Main} from Theorems \ref{NonUni} and ~\ref{MainRC}, using the maximal rationally connected fibration of $X$. We also prove Theorem \ref{MainKahler} in this section.

\textbf{Acknowledgement.} The author thanks Constantin Shramov for suggesting this problem; he also thanks Ignasi Mundet i Riera and the anonymous referee for valuable remarks. This work was supported by the Russian Science Foundation under grant no. 23-11-00033, https://rscf.ru/en/project/23-11-00033/

\section{Preliminaries}

\subsection{Conventions and terminology} An {\em algebraic variety} (or just a {\em variety}) is an integral separated scheme of finite type over a field $k$. Unless explicitly stated otherwise, the base field $k$ is always assumed to be algebraically closed and of characteristic zero. 

In what follows we consider irreducible and reduced compact complex spaces, see \cite[Chapter 1]{GR} for a general reference on complex analytic spaces. A complex {\em manifold} is a nonsingular complex space. We consider only compact K\"ahler manifolds. For the definition of a singular compact {\em K\"ahler space} see Definition \ref{singkahler}; this definition follows the one in \cite{HP16}.

\subsection{Structure of abelian subgroups} In this subsection we collect a few technical statements about subgroups of finite abelian groups.

\begin{definition} \label{rank} Let $G$ be a finite abelian group. The {\em rank} $r(G)$ is defined as the minimal size of a generating set of $G$. An {\em elementary} abelian group of rank $r$ is an abelian group isomorphic to $(\mathbb{Z}/N\mathbb{Z})^r$.
\end{definition}

\begin{lemma}\label{abelian} Let $G \simeq (\mathbb{Z}/N\mathbb{Z})^r$ be a finite abelian group. Let $H \subset G$ be a subgroup of index $$I_H \leqslant N-1.$$ Then there exists an elementary subgroup $H' \subset H$ such that $H' \simeq (\mathbb{Z}/N'\mathbb{Z})^r$ for some $N' \geqslant N/I_H$.
\end{lemma}

\begin{proof} Let $H \subseteq G$ be a subgroup of index $I_H \leqslant N-1$. Then $H$ is a finite abelian group of order at least ~$N(r-1) + 1$. The orders of generators of $H$ do not exceed $N$, so the rank of $H$ is at least $r$. By the structure theorem for finite abelian groups we have $$H \simeq \bigoplus_{1 \leqslant i \leqslant r}\mathbb{Z}/N_i\mathbb{Z}$$ where $N_i | N_{i+1}$ for all $i \in \{1, \ldots, r-1\}$. Next, from the equality $$|H| = N^r/I_H = N_1 \cdots N_r$$ we have $N_1 \geqslant N/I_H$. Now it suffices to take the elementary subgroup $H' = (\mathbb{Z}/N_1\mathbb{Z})^r \subseteq H$.
\end{proof}

\begin{lemma}\label{abelian3} Let $G \simeq (\mathbb{Z}/N\mathbb{Z})^r$ be a finite abelian group and let $H \subset G$ be a subgroup. There exist a set of generators $\{b_1, \ldots, b_r\}$ for $H$ and a set of generators $\{a'_1, \ldots, a'_r\}$ for $G$ such that the embedding ~$H \to G$ can be written as the direct sum of homomorphisms $$\mathbb{Z}/N_i\mathbb{Z} \to \mathbb{Z}/N\mathbb{Z}.$$
\end{lemma}

\begin{proof} We choose a set of generators $a_1, \ldots, a_r$ for $\mathbb{Z}^r$ such that their images under the projection $$\mathbb{Z}^r \to \mathbb{Z}^r/(N\mathbb{Z})^r$$ give an isomorphism $(\mathbb{Z}/N\mathbb{Z})^r \simeq G$. Let us denote by $\widetilde{H}$ the preimage of $H$ under the map $\mathbb{Z}^r \to G$. Then $\widetilde{H}$ is a free abelian group of rank $r$ containing the subgroup $(N\mathbb{Z})^r$. Let $A$ be a presentation matrix for $\widetilde{H}$. Using, for example,  \cite[Theorem (4.3)]{Art91} we find that there exists the Smith normal form $$A' = QAP^{-1}$$ where $Q, P \in \mathrm{GL}_r(\mathbb{Z})$ and the matrix $A'$ is diagonal with entries $N_1, \ldots, N_r$ such that $N_i$ divides $N_{i+1}$ for any ~$i \in \{1, \ldots, r-1\}$. So there exists a basis $a'_1, \ldots, a'_r$ of $\mathbb{Z}^r$ such that $\widetilde{H}$ is generated by $$b_1 = a'_1N_1, \ldots, b_r = a'_rN_r.$$ Since multiplication by invertible matrices preserves the sublattice $(N\mathbb{Z})^r \subset \mathbb{Z}^r$, we can take the images of $b_1, \ldots, b_r$ under the projection $\mathbb{Z}^r \to \mathbb{Z}^r/(n\mathbb{Z})^r \simeq G$ as generators for $H$.
\end{proof}

For the discussion that follows it will be convenient to introduce the following definition.

\begin{definition}\label{rank2} Let $\{G_i\}_{i \in \mathbb{N}}$ be a sequence of finite groups. We define the {\em asymptotic rank} of the sequence $\{G_i\}$ to be the minimal number $r$ such that the following condition is satisfied. There exists a constant $L$ such that, for infinitely many indices $i \in \mathbb{N}$, we can find an abelian subgroup $H_i \subset G_i$ such that \begin{itemize} \item $H_i$ is generated by $r$ elements; \item the orders of the subgroups $H_i$ are unbounded as $i$ tends to infinity; \item the index of $H_i$ in $G_i$ does not exceed $L$. \end{itemize}
\end{definition}

\begin{example} If the orders of the groups $G_i$ are bounded by a constant then the asymptotic rank of the sequence $\{G_i\}$ is equal to zero. The asymptotic rank of the sequence $G_i = (\mathbb{Z}/i\mathbb{Z})^{r} \times (\mathbb{Z}/2\mathbb{Z})^{10}$ is equal to $r$. The asymptotic rank of the sequence $G_r = (\mathbb{Z}/r\mathbb{Z})^r$ is infinite.
\end{example}

\begin{remark} The motivation for the above definition comes from the study of Jordan groups. Recall from \cite[Definition 2.1]{Pop11} that a group $G$ is Jordan if there exists a constant $J(G) \in \mathbb{N}$ such that for every finite subgroup $H \subset G$ there exists a normal abelian subgroup $A \lhd H$ of index at most $J(G)$. Suppose that the group $G$ is Jordan and that the orders of finite subgroups of $G$ are unbounded. Then there exists a sequence $\{G_i\}_{i \in \mathbb{N}}$ of finite subgroups of $G$ satisfying the assumptions in Definition \ref{rank2} for some $r$ and $L = J(G)$ (the Jordan constant of $G$). The maximum value of $r$ over all such sequences of finite subgroups of $G$ is a natural invariant of the group $G$.
\end{remark}

\begin{remark} The most basic example of a Jordan group is a linear algebraic group $G$ over an algebraically closed field $k$ of characteristic zero. In this case by \cite[Lemma 3.7]{Xu18} there exists a constant $B(n)$ such that, for every connected linear algebraic group $G$ over $k$ of rank at most $n$ and for every finite subgroup $H \subset G$, there exists a finite subgroup $N \subset H$ of index at most $B(n)$ such that $N$ is contained in a maximal torus of $G$. Thus, the asymptotic rank of any sequence of finite subgroups of $G$ is bounded from above by the rank of $G$.
\end{remark}

\begin{remark} \label{elementary} Let $\{G_i\}$ be a sequence of finite groups of asymptotic rank $r$. Then we can take a sequence of abelian subgroups $H_i \subset G_i$ as in Definition \ref{rank2}; then the  asymptotic rank of $\{H_i\}$ is equal to the asymptotic rank of $\{G_i\}$. More generally, if $\{G'_i \subset G_i\}$ is a sequence of subgroups of uniformly bounded index, then the asymptotic ranks of the sequences $\{G'_i\}$ and $\{G_i\}$ are equal. Therefore, by Lemma \ref{abelian}, it suffices to consider sequences of elementary abelian groups.
\end{remark}

For a sequence of finite abelian groups, the asymptotic rank can be computed using the direct sum decomposition provided by the structure theorem.

\begin{proposition} \label{rankab} Let $\{G_i\}_{i \in \mathbb{N}}$ be a sequence of finite abelian groups. Suppose that for every $i \in \mathbb{N}$ the rank of $G_i$ is $r$. Consider the decomposition $$G_i \simeq \bigoplus_{1 \leqslant k \leqslant r}\mathbb{Z}/N_{i,k}\mathbb{Z},$$ where $N_{i,k} | N_{i, k+1}$ for every $k \in \{1, \ldots, r-1\}$. Then the  asymptotic rank of the sequence $\{G_i\}$ is equal to $$r - \max\{k \mid \mbox{the sequence $\{N_{i, k}\}, i \in \mathbb{N}$ is bounded as $i \to \infty$}\}.$$
\end{proposition}

\begin{proof} We set $$k_{\max} = \max\{k \mid \mbox{the sequence $\{N_{i, k}\}, i \in \mathbb{N}$ is bounded as $i \to \infty$}\}.$$ 
Considering the sequence of subgroups $$H_i = \bigoplus_{k_{\max}+1 \leqslant k \leqslant r}\mathbb{Z}/N_{i,k}\mathbb{Z} \subset G_i,$$ we find that the asymptotic rank of the sequence $\{G_i\}$ is at most $r-k_{\max}$. We denote by $L$ a constant such that $|G_i|/|H_i| \leqslant L$ for all $i \in \mathbb{N}$.

Suppose that the asymptotic rank $r'$ of $\{G_i\}$ is smaller than $r - k_{\max}$. Then there exists a sequence of subgroups $\{H'_i \subset G_i\}$ such that for every $i \in \mathbb{N}$ the group $H'_i$ is generated by $r' < r - k_{\max}$ elements, and the indices $|G_i|/|H'_i|$ are bounded as $i \to \infty$. We have $$|H'_i| = |H'_i/(H_i \cap H'_i)|\cdot |H_i \cap H'_i| \leqslant L \cdot|H_i \cap H'_i|.$$ Hence $$\frac{|G_i|}{|H'_i|} \geqslant \frac{|G_i|}{L\cdot|H_i\cap H'_i|} = \frac{|G_i|}{L|H_i|}\cdot\frac{|H_i|}{|H_i\cap H'_i|}.$$ On the other hand, since the subgroup $H_i \cap H'_i$ is generated by $r' < r - k_{\max}$ elements, the indices $|H_i|/|H_i \cap H'_i|$ are unbounded as $i \to \infty$. This contradiction shows that the asymptotic rank of $\{G_i\}$ is equal to $r - k_{\max}$.
\end{proof}

Another convenient way to express the asymptotic rank of a sequence of finite abelian groups is given in the corollary below.

\begin{corollary}\label{rankab2} Let $\{G_i\}$ be a sequence of finite abelian groups. Suppose that for every $i \in \mathbb{N}$, the group $G_i$ can be generated by $r$ elements. Then the asymptotic rank of $\{G_i\}$ is equal to $$\max\{r \mid \mbox{$G_i \supset (\mathbb{Z}/M_i\mathbb{Z})^r$ for an infinite number of $i \in \mathbb{N}$ and $M_i \to \infty$}\}.$$
\end{corollary}

\begin{proof} By Proposition \ref{rankab}, the asymptotic rank of the sequence $\{G_i\}$ is $r - k_{\max}$, where $$k_{\max} = \max\{k \mid \mbox{the sequence $\{N_{i, k}\}, i \in \mathbb{N}$ is bounded as $i \to \infty$}\}.$$ Consider the sequence of subgroups $$H_i = \bigoplus_{k_{\max}+1 \leqslant k \leqslant r}\mathbb{Z}/N_{i,k}\mathbb{Z} \subset G_i.$$ Then each $H_i$ contains a subgroup isomorphic to $(\mathbb{Z}/M_i\mathbb{Z})^{r-k_{\max}}$ where $M_i = N_{i, k_{\max} + 1}$. Suppose that for infinitely many $i \in \mathbb{N}$ we can find subgroups $H'_i \subset G_i$ such that $H'_i \simeq (\mathbb{Z}/M'_i\mathbb{Z})^s$ for $M_i \to \infty$. Consider the images of $H'_i$ under the quotient homomorphisms $G_i \to G_i/H_i$. Since the indices $|G_i|/|H_i|$ are bounded by a constant $L$ independent of $i \in \mathbb{N}$, the number $s$ does not exceed $r - k_{\max}$.
\end{proof}

We deduce the following important subadditivity property for asymptotic ranks.

\begin{lemma}\label{abelianmain} Let $\{G_i\}$ be a sequence of abelian groups. Consider a sequence of subgroups $G'_i \subset G_i$ for $i \in \mathbb{N}$, and denote the quotient groups by $G''_i$. Suppose that the asymptotic rank of the sequence $\{G'_i\}$ is at most $r'$ and that the asymptotic rank of the sequence $\{G''_i\}$ is at most $r''$. Then the asymptotic rank $r$ of the sequence $\{G_i\}$ is at most $r'+r''$.
\end{lemma}

\begin{proof} By Remark \ref{elementary} we may assume that $G_i \simeq (\mathbb{Z}/N_i\mathbb{Z})^r$ are elementary abelian subgroups for some $N_i$ that tend to infinity. By Lemma \ref{abelian3} we may choose compatible systems of generators in $G'_i$ and $G_i$ for every $i \in \mathbb{N}$. Let us denote by $$N'_{i, 1} | N'_{i, 2} | \cdots | N'_{i, r}$$ the divisors in the decomposition of $G'_i$ given by the structure theorem. Then the quotient groups $G''_i$ are isomorphic to $$\bigoplus_{1 \leqslant j \leqslant r}\mathbb{Z}/\frac{N_i}{N'_{i,j}}\mathbb{Z}.$$ By the assumption, the asymptotic rank of the sequence $\{G'_i\}$ is at most $r'$. By Proposition \ref{rankab}, $$r - \max\{j \mid \mbox{the sequence $\{N_{i, j}\}$ is bounded as $i \to \infty$}\} \leqslant r'.$$ Similarly, since the asymptotic rank of the sequence $\{G''_i\}$ is at most $r''$, we have, by Proposition \ref{rankab}, $$r - \min\{j \mid \mbox{the sequence $\{\frac{N_i}{N'_{i, j}}\}$ is bounded as $i \to \infty$}\} \leqslant r''.$$ However, since $N_i$ tend to infinity, the sequences $\{N'_{i,j}\}$ and $\{N_i/N'_{i,j}\}$ for a fixed $j \in \{1, \ldots, r\}$ cannot be bounded simultaneously. Therefore, adding the above inequalities we obtain $$r \leqslant r' + r'',$$ the result required.
\end{proof}

\subsection{The MRC fibration} 

In this subsection we briefly recall the construction of the maximal rationally connected (MRC) fibration for a compact K\"ahler manifold $X$. In this generality, the existence of the MRC fibration was established in \cite{Cam04}. We refer to \cite{Cam04} for details, including the definition of the cycle space (Barlet space) $\mathcal{C}(X)$ for a compact complex space $X$. For a purely algebraic proof of this result in the case of projective algebraic varieties, see \cite[Th\'eor\`eme 2.3]{Cam92} or \cite{KMM92}.

\begin{definition} A covering family of cycles on a complex space $X$ is a complex subspace $S \subset \mathcal{C}(X)$ such that \begin{itemize} \item $S$ is a countable union of compact irreducible complex subspaces; \item For $s \in S_i$ a general point the cycle $Z_s$ is irreducible and reduced; \item $X$ is a union of $\mathrm{Supp}(Z_s)$ for $s \in S$.\end{itemize}
\end{definition} 

A covering family of cycles induces an equivalence relation $R(S)$ on points of $X$. Namely, two points $x, y \in X$ are equivalent if and only if they are contained in a connected union of a finite number of cycles parameterized by $S$. The following theorem (see \cite[Theorem 1.1]{Cam04} for the proof) shows the existence of meromorphic reduction maps for covering families of cycles. Recall that a {\em fibration} is a dominant meromorphic map of normal complex spaces with connected fibers. A {\em typical fiber} of a fibration is a fiber over a point in the complement to a proper analytic subset in the base.

\begin{theorem}\label{mrc} Let $X$ be a normal compact connected complex space. Let $S \subset \mathcal{C}(X)$ be a covering family of cycles on $X$. Denote by $R(S)$ the equivalence relation on $X$ induced by $S$. Then there exists a meromorphic fibration $q_S \colon X \dasharrow B_S$ such that a typical fiber of $q_S$ is an equivalence class for $R(S)$. 
\end{theorem}

An important result, proved independently in \cite{Fuj78} and \cite{Lie78}, is the compactness of the irreducible components of $\mathcal{C}(X)$ in the K\"ahler case.

\begin{theorem} \label{compact} Let $X$ be a compact K\"ahler manifold. Then each irreducible component of the cycle space $\mathcal{C}(X)$ is compact.
\end{theorem}

As a consequence of Theorem \ref{compact}, for a compact K\"ahler manifold one can define the following natural meromorphic fibration.

\begin{definition}\label{mrcdef} Let $X$ be a compact K\"ahler manifold and let $S$ be a family of all rational curves on ~$X$. The fibration $f \colon X \dasharrow B$ corresponding to $S$ by Theorem \ref{mrc} is called the maximal rationally connected (MRC) fibration of $X$. 
\end{definition}

Obviously, if $X$ is not covered by rational curves then $f$ is birational. A crucial property of the MRC fibration is that its base $B$ is not covered by rational curves. This statement was shown in \cite[Corollary ~1.4]{GHS03} for $X$ an algebraic variety. The same argument generalizes to the K\"ahler case (see, for instance, \cite[Remark 3.2]{HP15} or \cite[Proposition 3.8]{PS20}).

\begin{theorem}\label{ghs} Let $X$ be a compact K\"ahler manifold. Consider the MRC fibration $$f \colon X \dasharrow B.$$ Then the base $B$ is not uniruled.
\end{theorem}

Moreover, the smooth fibers of the MRC fibration of a compact K\"ahler manifold $X$ are in fact projective, see \cite[Theorem 3.9]{PS20} for the proof.

\begin{proposition} \label{RatKah} Let $X$ be a rationally connected compact K\"ahler manifold. Then $X$ is projective.
\end{proposition}

\subsection{Finite group actions}

We will need a well-known result (see e. g. \cite[Lemma 3.1]{PS14}) on existence of regularizations of birational actions of finite groups.

\begin{proposition}\label{reg} Let $X$ be a normal projective variety and let $G \subset \Bir(X)$ be a finite group. Then there exists a smooth projective variety $\overline{X}$ with a regular action of $G$ and a $G$-equivariant birational map $$\varphi \colon \overline{X} \dasharrow X.$$
\end{proposition}

\begin{proof} 
Replacing $X$ by an affine open subset, we may assume that the action of $G$ on $X$ is regular. Then by \cite[Theorem 3]{Sum74}, there exists a $G$-equivariant projective completion $$\varphi \colon \overline{X} \dasharrow X.$$ Replacing $\overline{X}$ by a $G$-equivariant resolution of singularities of $\overline{X}$ (see e. g. \cite{BM97}) we may assume $\overline{X}$ to be smooth.
\end{proof}

The following proposition shows that actions of finite groups by automorphisms can be linearized in the fixed points. For the proof in the complex analytic setup we refer to  \cite[p. 38]{Akh95}.

\begin{proposition}\label{faith} Let $G$ be a finite group acting on a compact complex space $X$ by biholomorphic automorphisms with a fixed point $p \in X$. Then the induced action of $G$ on the tangent space $T_p(X)$ is faithful. 
\end{proposition}

\section{Main results}

\subsection{Groups of pseudoautomorphisms}

In this section we extend Theorem \ref{mundet1} to automorphism groups of singular compact K\"ahler spaces. For a complex space $X$ we denote the subsets of its singular and non-singular points by $X_{\mathrm{sing}}$ and $X_{\mathrm{ns}}$, respectively.

\begin{definition} \label{singkahler} Let $X$ be an irreducible and reduced complex space. A {\em K\"ahler form} on $X$ is a closed positive real $(1,1)$-form $\omega$ on $X_{\mathrm{ns}}$ satisfying the following condition: for any $x \in X_{\mathrm{sing}}$ there exists an open neighborhood $x \in U \subset X$ with a closed embedding $i_U \colon U \subset V$ into an open subset $V \subset \mathbb{C}^N$ such that $$\omega|_{U \cap X_{\mathrm{ns}}} = i\partial\bar\partial f|_{U \cap X_{\mathrm{ns}}}$$ for a smooth strictly plurisubharmonic function $f \colon V \to \mathbb{C}$. An irreducible and reduced complex space $X$ is {\em K\"ahler} if there exists a K\"ahler form on $X$.
\end{definition}

\begin{remark} Below we consider only those singular K\"ahler spaces that are normal and have rational singularities. In particular, minimal and quasi-minimal compact K\"ahler spaces (or complex projective varieties) satisfy these conditions.
\end{remark}

\begin{remark} If $X$ is a singular K\"ahler space, one can always find a resolution of singularities $\varphi \colon X' \to X$ where $X'$ is a compact K\"ahler manifold \cite[Remark 2.3]{HP16}. The MRC fibration for $X$ can be defined as the MRC fibration of (any) compact K\"ahler manifold $X'$ bimeromorphic to $X$.
\end{remark}

We need the following simple lemma (cf. \cite[Lemma 9.11]{Ue75}).

\begin{lemma}\label{torus} Let $X$ be a normal compact K\"ahler space. Suppose that there exists a bimeromorphic morphism $$\varphi \colon T \to X,$$ where $T$ is a compact complex torus. Then $\varphi$ is an isomorphism.
\end{lemma}

\begin{proof} Let $E$ be an irreducible component of the exceptional locus of $\varphi$ of dimension $c > 0$. Consider a K\"ahler class $\omega$ on $X$. By the projection formula, $$(\varphi^*\omega)^c \cdot E = \omega^c \cdot (\varphi_*E) = 0.$$ On the other hand, we can choose a general translation $\tau \colon T \to T$ such that the image of $\tau^*(E)$ under the map $\varphi$ is not contained in the singular locus of $X$. Therefore, $$(\varphi^*\omega)^c \cdot \tau^*E = \omega^c \cdot (\varphi_*\tau^*E) > 0.$$ However, since $\tau$ is an automorphism of $T$, we have $(\varphi^*\omega)^c \cdot E = (\varphi^*\omega)^c \cdot (\tau^*E)$. This contradiction shows that the exceptional locus of $\varphi$ is empty, so $\varphi$ is an isomorphism.
\end{proof}

We also state another result by I. Mundet i Riera (see \cite[Theorem 1.10]{Mun21}). Theorem \ref{mundet1} is immediate from this result.

\begin{theorem}\label{mundet} Let $G$ be a Lie group with finitely many connected components. For every natural number $r$, the following properties are equivalent: \begin{itemize} \item the group $G$ contains subgroups of the form $(\mathbb{Z}/N\mathbb{Z})^r$ for arbitrarily large positive integers $N$; \item the group $G$ contais a subgroup isomorphic to a compact real torus $(S^1)^{r}$ of real dimension $r$. \end{itemize}
\end{theorem}

We can deduce the following corollary from Theorem \ref{mundet} by an argument similar to the proof of \cite[Theorem 1.9]{Mun21}.

\begin{corollary}\label{singaut} Let $X$ be a (possibly singular) normal compact K\"ahler space. For every natural number $r$, the following properties are equivalent: \begin{itemize} \item the group $\Aut(X)$ contains finite abelian subgroups of the form $(\mathbb{Z}/N\mathbb{Z})^r$ for arbitrarily large positive integers $N$; \item the group $\Aut(X)$ contais a subgroup isomorphic to $(S^1)^r$. \end{itemize}
In addition, $r \leqslant 2\dim(X)$, and if $r = 2\dim(X)$ then $X$ is biholomorphic to a compact complex torus.
\end{corollary}

\begin{proof} The group of connected components $\Aut(X)/\Aut^0(X)$ has bounded finite subgroups (see \cite[Lemma 3.1]{Kim18}). So, by Lemma \ref{abelian}, we may assume that the finite abelian subgroups in question lie in $\Aut^0(X)$. By a well-known theorem of S. Bochner and H. Montgomery, the group $\Aut^0(X)$ is a connected complex Lie group acting holomorphically on $X$ (see e. g. \cite[Theorem on p. 40]{Akh95} for a modern proof). Now the first statement of the corollary follows from Theorem \ref{mundet}. 

Suppose that there is an effective action of $(S^1)^r$ on $X$ by holomorphic automorphisms. Then by the results of \cite{BM97}, we can take a $(S^1)^r$-equivariant resolution of singularities $\varphi \colon X' \to X$, where $X'$ is a compact K\"ahler manifold. Applying Theorem \ref{mundet1} to $X'$, we get the estimate $$r \leqslant 2\dim(X') = 2\dim(X).$$ Now, if $r = 2\dim(X')$, then $X'$ is biholomorphic to a compact complex torus. By Lemma \ref{torus}, $\varphi$ is an isomorphism, and so $X$ is nonsingular and biholomorphic to a compact complex torus.
\end{proof}

The next step is to extend the above result to groups of pseudoautomorphisms of singular compact K\"ahler spaces. Recall that a bimeromorphic map $f \colon X \dasharrow X$ is a {\em pseudoautomorphism} if both $f$ and $f^{-1}$ do not contract divisors. The group of pseudoautomorphisms of $X$ is denoted by $\Psaut(X)$.

For convenience of the reader, we reproduce here the following result (see \cite[Corollary 4.6]{Gol23}). 

\begin{proposition}\label{psautreg} Let $X$ be a normal compact K\"ahler space with rational singularities. Let $f \colon X \dasharrow X$ be a pseudoautomorphism. Suppose that there exists a K\"ahler class $\omega$ such that $f_*\omega$ is also a K\"ahler class. Then $f$ is a biholomorphic automorphism of $X$.
\end{proposition}

Now using this proposition we can easily generalize Corollary \ref{singaut} to the group $\Psaut(X)$.

\begin{theorem}\label{MainThm2} Let $X$ be a normal compact K\"ahler space with rational singularities. Suppose that there exists $r \in \mathbb{N}$ such that the group $\Psaut(X)$ contains finite abelian  subgroups isomorphic to $(\mathbb{Z}/N\mathbb{Z})^r$ for arbitrarily large ~$N$. Then $r \leqslant 2\dim(X)$ and $\Psaut(X)$ contains a subgroup isomorphic to a compact real torus $(S^1)^r$. In addition, if $r = 2\dim(X)$, then $X$ is biholomorphic to a compact complex torus.
\end{theorem}

\begin{proof} As in the proof of \cite[Theorem 4.5]{Gol23}, we consider the action of $\Psaut(X)$ on $H^{2}(X, \mathbb{Q})$ by pushforward. We have an exact sequence of groups $$1 \to \Psaut(X)_{\tau} \to \Psaut(X) \to \Psaut(X)/\Psaut(X)_{\tau} \to 1,$$ where we set $$\Psaut(X)_{\tau} = \{f \in \Psaut(X) \mid f_*|_{H^{2}(X, \mathbb{Q})} = \mathrm{Id}\}.$$ Note that the quotient group $\Psaut(X)/\Psaut(X)_{\tau}$ embeds into $\mathrm{GL}(H^{2}(X, \mathbb{Q}))$, therefore, by Minkowski's theorem (see e.g. \cite[Theorem 1]{Ser07}) the orders of finite subgroups of $\Psaut(X)/\Psaut(X)_{\tau}$ are bounded by a constant $M(X)$ depending on $h^{2}(X, \mathbb{Q})$ only. Hence the group $\Psaut(X)_{\tau}$ contains a sequence of finite abelian subgroups of asymptotic rank $r$; in addition, by Lemma \ref{abelian}, we may assume that these subgroups are of the form $(\mathbb{Z}/N_i\mathbb{Z})^r$, where $N_i$ tend to infinity. The group $\Psaut(X)_{\tau}$ acts trivially on $H^{2}(X, \mathbb{R}) = H^{2}(X, \mathbb{Q}) \otimes_{\mathbb{Q}} \mathbb{R}$ and, in particular, it preserves every K\"ahler class on $X$. Thus by Proposition \ref{psautreg}, the group $\Psaut(X)_{\tau}$ is contained in $\Aut(X)$. The theorem now follows from Corollary \ref{singaut}.
\end{proof}

\subsection{Non-uniruled varieties and complex spaces} In this subsection we use Theorem \ref{MainThm2} to derive a slightly more general version of Theorem ~\ref{jxu1} from the Introduction.  

To define minimal and quasi-minimal models of compact K\"ahler spaces, we need to introduce notions of nefness and modified nefness in the non-projective context (see \cite{Bou04, HP16} for more details).

\begin{definition}\label{Nef} Let $X$ be a normal compact K\"ahler space with rational singularities. We say that a class $\alpha \in H^{1,1}(X, \mathbb{R})$ is \begin{itemize} \item {\em nef} if it belongs to the closure of the cone of K\"ahler classes;
\item {\em modified nef} if it belongs to the closure of the cone generated by classes of the form $\mu_*\omega$ where $\mu \colon Y \to X$ is an arbitrary bimeromorphic morphism from a smooth compact K\"ahler manifold $Y$ and $\omega$ is a K\"ahler class on $Y$.
\end{itemize}
\end{definition}

\begin{definition} \label{qmin} A compact K\"ahler space (or a projective variety) $X$ with terminal $\mathbb{Q}$-factorial singularities is called \begin{itemize} \item {\em minimal} (or a {\em minimal model}) if the canonical class $K_X$ is nef; \item {\em quasi-minimal} (or a {\em quasi-minimal model}) if $K_X$ is modified nef. \end{itemize}  
\end{definition}

Note that a minimal model is also quasi-minimal. Existence of quasi-minimal models for non-uniruled projective varieties was shown in \cite[Lemma ~4.4]{PS14}.

\begin{proposition} \label{qmin} Let $X$ be a non-uniruled projective variety. Then there exists a quasi-minimal model of $X$, that is, a quasi-minimal variety $X'$ birational to $X$.
\end{proposition}

In the case of non-uniruled compact K\"ahler spaces of dimension 3, minimal models exist by \cite[Theorem 1.1]{HP16}.

\begin{theorem}\label{hp} Let $X$ be a compact K\"ahler space of dimension 3. Then there exists a minimal compact K\"ahler space $X'$ bimeromorphic to $X$.
\end{theorem}

The reason to consider quasi-minimal models is the following description of their bimeromorphic (or birational) automorphisms. The case when $X$ is a projective variety was settled in \cite[Corollary 4.7]{PS14}; for the general case of compact K\"ahler spaces see \cite[Proposition 4.2]{Gol23}.

\begin{proposition}\label{psaut} Let $X$ be a quasi-minimal compact K\"ahler space. Let $f \colon X \dasharrow X$ be a bimeromorphic map. Then $f$ is a pseudoautomorphism.
\end{proposition}

Now we can prove Theorem \ref{Main} for a non-uniruled projective variety $X$ over an algebraically closed field $k$ of zero characteristic. Without loss of generality we may assume that $k = \mathbb{C}$. 

\begin{theorem}\label{NonUni} Let $X$ be a non-uniruled projective variety over the field of complex numbers. Suppose that there exists $r \in \mathbb{N}$ such that the group $\Bir(X)$ contains finite abelian subgroups isomorphic to $(\mathbb{Z}/N\mathbb{Z})^r$ for arbitrarily large positive integers $N$. Then $$r \leqslant 2\dim(X),$$ and the group $\Bir(X)$ contains a subgroup isomorphic to an abelian variety of dimension ~$\lceil r/2\rceil$. In the case $r = 2\dim(X)$ the variety $X$ is birational to an abelian variety.
\end{theorem}

\begin{proof} Since $X$ is non-uniruled, by Proposition \ref{qmin} there exists a quasi-minimal projective variety $X'$ birational to ~$X$. By Proposition \ref{psaut}, we have $\Bir(X) \simeq \Bir(X') = \Psaut(X')$. The upper bound $r \leqslant 2\dim(X)$ now follows from Theorem \ref{MainThm2}. Since $X$ is not covered by rational curves, the compact real torus $(S^1)^r$ in the connected component $\Aut^0(X)$ can only be contained in an abelian variety of complex dimension at least $\lceil r/2\rceil$.
\end{proof}

An analogous result holds for a compact K\"ahler space $X$, under the assumption that a quasi-minimal model of $X$ exists. By Theorem \ref{hp} this condition holds if $\dim(X) \leqslant 3$.

\begin{proposition}\label{NonUniKah} Let $X$ be a non-uniruled compact K\"ahler space admitting a quasi-minimal model. Let, for some $r \in \mathbb{N}$, the group $\Bim(X)$ contain finite abelian subgroups isomorphic to $(\mathbb{Z}/N\mathbb{Z})^r$ for arbitrarily large $N$. Then $r \leqslant 2\dim(X)$, and the group $\Bim(X)$ contains a subgroup isomorphic to a compact complex torus of dimension ~$\lceil r/2\rceil$. In addition, if $r = 2\dim(X)$, then $X$ is bimeromorphic to a compact complex torus.
\end{proposition}

\begin{proof} By the assumption, there exists a quasi-minimal compact K\"ahler space $X'$ bimeromorphic to $X$. By Proposition \ref{psaut}, $\Bim(X) \simeq \Bim(X') = \Psaut(X')$. Now the required result is secured by Theorem ~\ref{MainThm2}.
\end{proof}

\subsection{Rationally connected varieties}

We recall an important result on boundedness for finite groups acting on rationally connected algebraic varieties.

\begin{proposition}\label{fixpt} Let $X$ be a rationally connected algebraic variety of dimension $n$ over an algebraically closed field $k$ of zero characteristic. Then there exists a constant $J(n)$, depending on $n$ only, such that for any finite subgroup $G \subseteq \Aut(X)$ there exists a subgroup $H \subseteq G$ of index at most $J(n)$ acting on $X$ with a fixed point.
\end{proposition} 

This result is immediate from \cite[Theorem 4.2]{PS16} and \cite[Theorem 1.1]{Bir21}. 

As a result, we have the following upper bound for ranks of finite abelian subgroups in the group $\Bir(X)$, where $X$ is a geometrically rationally connected algebraic variety over any field of zero characteristic.

\begin{corollary} \label{constant1} Let $X$ be a geometrically integral and geometrically rationally connected algebraic variety of dimension $n$ over an arbitrary field $k$ of zero characteristic. There exists a constant $M = M(n)$ such that, for any finite subgroup $G \subset \Bir(X)$, there exists an abelian subgroup $H \subset G$ of index at most $M(n)$ and such the rank of $H$ does not exceed $n$.
\end{corollary}

\begin{proof} We pass to the algebraic closure $\overline{k}$ of $k$ and replace $X$ by $X \times_{k}\overline{k}$. Let $\varphi \colon X' \dasharrow X$ be a smooth birational regularization of the action of $G$, which exists by Proposition \ref{reg}. Note that $X'$ is rationally connected as well. Therefore by Proposition \ref{fixpt} there exists a constant $J'(n)$ such that $G$ contains a subgroup $H$ of index at most $J'(n)$ acting on $X'$ with a fixed point. By Proposition \ref{faith} the group $H$ embeds into ~$\mathrm{GL}_n(\overline{k})$. Therefore by Jordan's theorem (see \cite{Jor78} or \cite{Rob90}) there exists a constant $J''(n)$ such that the group $H$ contains an abelian subgroup $A \subset H$ of index at most $J''(n)$. Then $A \subset G$ is a subgroup of index at most $$M(n) = J'(n)\cdot J''(n),$$ moreover, since $A$ is linear, it is generated by at most $n$ elements.
\end{proof}

Let us now show that any sequence of finite abelian subgroups in $\Bir(X)$ has asymptotic rank at most $n$ (see Corollary \ref{rankab2}). 

\begin{theorem}\label{MainRC} Let $X$ be a geometrically integral and geometrically rationally connected algebraic variety of dimension $n$ over an arbitrary field $k$ of zero characteristic. Suppose that there exists an unbounded sequence $\{N_i\}_{i \in \mathbb{N}}$ of positive integers such that the group $\Bir(X)$ contains a subgroup $G_i \simeq (\mathbb{Z}/N_i\mathbb{Z})^r$ for some fixed $r \in \mathbb{N}$. Then $r \leqslant n$.
\end{theorem}

\begin{proof} By Proposition 3, there exists a constant $M(n)$ such that, for each $i \in \mathbb{N}$,
there exists an abelian subgroup $H_i \subset G_i$ of index $\leqslant M(n)$. Hence the asymptotic
rank of $\{G_i\}$ is equal to that of $\{H_i\}$, which is at most $n$, because all finite abelian
groups $H_i$ are of rank $\leqslant n$. This proves the theorem.
\end{proof}

\subsection{The general case} Now we can prove Theorems \ref{Main} and \ref{MainKahler} from the Introduction.

\begin{proof}[Proof of Theorem 1.5] Passing to a resolution of singularities, we may assume that $X$ is smooth. If $X$ is not uniruled then the result follows from Theorem \ref{NonUni}. Suppose that $X$ is uniruled and consider its MRC fibration $f \colon X \dasharrow B$, where $\dim(B) < \dim(X)$. Then for every $i \in \mathbb{N}$ we have the exact sequence $$1 \to G'_i \to G_i \to G''_i \to 1,$$ where the action of $G'_i$ is fiberwise with respect to $f$ (that is, every element $g \in G'_i$ maps a point in a fiber of $f$ where $g$ is defined to a point in the same fiber) and $G''_i$ acts faithfully on the base $B$. Let $X_{\eta}$ be the scheme-theoretic generic fiber of $f$. Then for every $i \in \mathbb{N}$ we have $G'_i \subset \Bir(X_{\eta})$. We denote ~$n' = \dim(X_{\eta})$. Then by Theorem \ref{MainRC} the asymptotic rank of the sequence $\{G'_i\}$ is at most ~$n'$. Since $B$ is not uniruled by Proposition \ref{ghs}, we apply Theorem \ref{NonUni} and obtain that the asymptotic rank of the sequence $\{G''_i\}$ is at most $2\dim(B)$. Therefore by Lemma \ref{abelianmain}, the asymptotic rank $r$ of the sequence $\{G_i\}$ is at most $n' + 2\dim(B)$. In particular, $$r \leqslant 2\dim(B) + n' = \dim(X) + \dim(B) < 2n,$$ the result required.
\end{proof}

It is also possible to describe projective varieties such that $\Bir(X)$ contains a sequence of finite abelian groups of submaximal asymptotic rank.

\begin{corollary} Let $X$ be a projective variety over an algebraically closed field of zero characteristic. Suppose that there exists an unbounded sequence $\{N_i\}_{i \in \mathbb{N}}$ of positive integers such that the group $\Bir(X)$ contains subgroups $G_i$ isomorphic to $(\mathbb{Z}/N_i\mathbb{Z})^r$ for $r = 2\dim(X) - 1$. Then $X$ is birational either to \begin{itemize} \item an abelian variety $A$; \item the product $\PP^1 \times A$ where $A$ is an abelian variety of dimension $\dim(X)-1$. \end{itemize}
\end{corollary}

\begin{proof} For $\dim(X) = 1$ the result is obvious, since $X$ is then isomorphic to a rational or elliptic curve. Suppose from now on that $\dim(X) > 1$. If $X$ is not uniruled, then by Theorem \ref{NonUni} there exists a birational model $X'$ of $X$ and a faithful action of an abelian variety of dimension $$\left\lceil \frac{2\dim(X') - 1}{2} \right\rceil = \dim(X') = \dim(X)$$ on $X'$, so that by Theorem \ref{Main} $X$ is birational to an abelian variety. 

Suppose now that $X$ is uniruled and consider the MRC fibration $f \colon X \dasharrow B$. Since $\dim(X) > 1$, we have $2\dim(X) - 1 > \dim(X)$ and therefore by Theorem \ref{MainRC} $X$ cannot be rationally connected, that is, $\dim(B) > 0$. Let $X_{\eta}$ be the general fiber of $f$. Let also $\{G'_i\}$ be the sequence of subgroups acting fiberwise with respect to $f$, and denote by $\{G''_i\}$ the sequence of quotient groups. By Theorem \ref{MainRC} the asymptotic rank of $\{G'_i\}$ does not exceed $\dim(X_{\eta}) \leqslant \dim(X)$. Therefore the asymptotic rank of $\{G_i''\}$ is at least $$2\dim(X) - 1 - \dim(X_{\eta}) \geqslant 2\dim(B) > 0.$$ Now by Theorem \ref{NonUni} the non-uniruled variety $B$ is birational to an abelian variety $A$; moreover, $A$ has maximal possible dimension, equal to $\dim(X) - 1$. Since the asymptotic rank of $\{G_i'\}$ is equal to 1, it follows by \cite[Theorem 4.14]{BZ17} that $X_\eta \simeq \PP^1_{k(B)}$, and so $X$ is birational to a product $\mathbb{P}^1\times A$.
\end{proof}

Before proceeding with the proof of Theorem \ref{MainKahler}, we need the following technical lemma (see \cite[Lemma 3.1]{PS21} for the proof). Recall that a {\em very typical} fiber of a dominant meromorphic map $\alpha \colon X \dasharrow Y$ is a fiber $X_t = \alpha^{-1}(t)$ over a point $t \in Y$ in the complement to at most countable union of proper analytic subspaces of $Y$. By $\Bim(X)_{\alpha}$ we denote the subgroup of elements of $\Bim(X)$ acting fiberwise with respect to $\alpha$.

\begin{lemma}\label{typfiber} Let $\alpha \colon X \dasharrow Y$ be a dominant meromorphic map of compact complex manifolds. Then there exist a constant $I = I(\alpha)$ with the following property. Let $\{G_i\}_{i \in \mathbb{N}}$ be a sequence of finite subgroups of $\Bim(X)_{\alpha}$. Then there exists a reduced fiber $F$ of $\alpha$ and its irreducible component $F'$ of dimension $\dim(X) - \dim(Y)$ such that for every $i \in \mathbb{N}$ the group $G_i$ contains a subgroup of index at most $I$, which is isomorphic to a subgroup of $\Bim(F')$. Moreover, if $\dim(Y) > 0$ the fiber $F$ can be chosen to be very typical. 
\end{lemma}

Now we can apply the same line of reasoning to the case of compact K\"ahler spaces, applying Lemma \ref{typfiber} to the MRC fibration of $X$. 

\begin{proof}[Proof of Theorem 1.6] Passing to a resolution of singularities, we may assume that $X$ is smooth. If $X$ is not uniruled then the result follows from Proposition \ref{NonUniKah}. 

Suppose that $X$ is uniruled. Then we consider the MRC fibration $f \colon X \dasharrow B$ with $B$ non-uniruled and $\dim(B) < \dim(X)$. If $\dim(B) = 0$ then $X$ is rationally connected and hence projective by Proposition \ref{RatKah}; this case follows from Theorem \ref{MainRC}. Assume from now on that $\dim(B) > 0$. Then for every $i \in \mathbb{N}$ there exists an exact sequence of groups $$1 \to G'_i \to G_i \to G''_i \to 1,$$ where the action of $G'_i$ is fiberwise with respect to $f$ and $G''_i$ acts faithfully on $B$. Since the set of finite groups $\{G_i\}_{i \in \mathbb{N}}$ is countable, by Lemma \ref{typfiber} we may assume that $$G'_i \subset \Bim(X_t),$$ where $X_t$ is a very typical (in particular, smooth) fiber of $f$. Note that by Proposition \ref{RatKah} smooth fibers of $f$ are projective. Now by Theorem \ref{MainRC}, the asymptotic rank of the sequence $\{G'_i\}$ is at most $\dim(X_t)$. Moreover, by the assumptions on $B$ and by Theorem ~\ref{NonUni} the asymptotic rank of the sequence $\{G''_i\}$ is at most $2\dim(B)$. By Lemma ~\ref{abelianmain}, the asymptotic rank $r$ of the sequence $\{G_i\}$ is at most $2\dim(B) + \dim(X_t)$; in particular, $$r \leqslant \dim(X_t) + 2\dim(B) < 2\dim(X),$$ as desired.
\end{proof}

\begin{remark} To prove Theorem \ref{MainKahler} in full generality, it suffices to prove that ``large'' finite abelian subgroups of $\Bim(X)$ can be pseudo-regularized on a compact K\"ahler manifold $X'$ bimeromorphic to $X$. By considering the algebraic reduction (see \cite[Definition 3.3]{Ue75}) of a compact K\"ahler space $X$, it suffices to resolve the above problem in the case when $X$ has algebraic dimension 0. In particular, if $X$ has no divisors (like a general compact complex torus) this statement is clear, since $\Bim(X) = \Psaut(X)$ in this case.
\end{remark}

\medskip

\flushleft{\address{Steklov Mathematical Institute of Russian Academy of Sciences, Moscow, Russia \\
8 Gubkina St., Moscow 119991, Russia \\}}
\email{golota.g.a.s@mail.ru, agolota@hse.ru}

\end{document}